\documentclass[a4paper, 12pt]{article}

\usepackage[utf8]{inputenc}
\usepackage[T1]{fontenc}
\usepackage[a4paper, margin=2.5cm]{geometry}
\usepackage{needspace}
\usepackage{mathtools}
\usepackage[thinc]{esdiff}

\usepackage{amsmath, amsthm, amssymb}
\usepackage{graphicx}
\usepackage{enumerate}
\usepackage{authblk}
\usepackage[textsize=small, textwidth=2cm, color=yellow]{todonotes}
\usepackage{thmtools}
\usepackage{thm-restate}
\usepackage[colorlinks=true, linkcolor =blue!80, citecolor=orange!70!black]{hyperref}
\usepackage{float}

\usepackage[capitalize]{cleveref}
\usepackage{comment}

\renewenvironment{abstract}
{\small\vspace{-1em}
\begin{center}
\bfseries\abstractname\vspace{-.5em}\vspace{0pt}
\end{center}
\list{}{
\setlength{\leftmargin}{0.6in}%
\setlength{\rightmargin}{\leftmargin}}%
\item\relax}
{\endlist}

\declaretheorem[name=Theorem, numberwithin=section]{theorem}
\declaretheorem[name=Lemma, sibling=theorem]{lemma}

\declaretheorem[name=Observation, sibling=theorem]{observation}

\def\cqedsymbol{\ifmmode$\lrcorner$\else{\unskip\nobreak\hfil
\penalty50\hskip1em\null\nobreak\hfil$\lrcorner$
\parfillskip=0pt\finalhyphendemerits=0\endgraf}\fi}

\DeclareMathOperator{\tw}{\mathsf{tw}}

\DeclareMathOperator{\polylog}{\mathrm{polylog}}

\newcommand{\N}{\mathbb{N}}

\let\le\leqslant
\let\ge\geqslant
\let\leq\leqslant
\let\geq\geqslant

\title{Treewidth versus clique number: induced minors}

\author[1,2]{Claire Hilaire}

\author[1,2]{Martin Milanič}

\author[3]{Nicolas Trotignon}

\author[1]{Djordje Vasić}

\affil[1]{FAMNIT, University of Primorska, Koper, Slovenia}

\affil[2]{IAM, University of Primorska, Koper, Slovenia}

\affil[3]{CNRS, ENS de Lyon, Université Claude Bernard Lyon 1, LIP~UMR~5668, 69342~Lyon~Cedex~07, France}

\date{}

\begin{document}

\maketitle

\bibliographystyle{abbrv}

\begin{abstract}
    We prove that a hereditary class of graphs is $(\tw, \omega)$-bounded if and only if the induced minors of the graphs from the class form a $(\tw, \omega)$-bounded class. 
\end{abstract}

\section{Introduction}

It is conjectured in \cite{DBLP:journals/jctb/DallardMS24} that if a hereditary class of graphs is $(\tw, \omega)$-bounded, then it has  bounded tree-independence number (all definitions are given in \cref{sec:def}).  This was disproved in~\cite{chuTro24} by an explicit counterexample. Proving weaker statements is therefore of interest. Here we investigate a weakening obtained by strengthening the assumption. Instead of assuming that the class is $(\tw, \omega)$-bounded, we make the formally stronger assumption that the induced minors of the graphs in the class form a $(\tw, \omega)$-bounded class.  

The goal of this note is to observe that this seemingly stronger assumption is in fact not stronger, because of our main result: a hereditary class of graphs is  $(\tw, \omega)$-bounded if and only if the induced minors of the graphs in the class form a  $(\tw, \omega)$-bounded class of graphs. 
This result is proved in \cref{sec:main}. 
Some additional remarks are given in \cref{sec:open}.

\section{Preliminaries}
\label{sec:def}

We denote by $\N$ the set of nonnegative integers, $\N = \{0,1,2,\ldots\}$.
All the graphs considered in this paper are finite, simple (that is, without loops or parallel edges), and undirected.
The \emph{subdivision of an edge} $uv$ of a graph
is the operation that removes the edge $uv$ and adds two edges $uw$ and $wv$, where $w$ is a new vertex.
A \emph{subdivision of a graph} $H$ is a graph obtained
from $H$ by a sequence of edge subdivisions.
An \emph{independent set} in a graph $G$ is a set of pairwise nonadjacent vertices.
The \emph{independence number} of a graph $G$, denoted $\alpha(G)$, is the size of a largest independent set in $G$.
A \emph{clique} in a graph $G$ is a set of pairwise adjacent vertices.
The \emph{clique number} of a graph $G$, denoted $\omega(G)$, is the size of a largest clique in $G$.
Given a set $S\subseteq V(G)$, we denote by $G[S]$ the subgraph of $G$ induced by $S$.
For $v\in V(G)$, the set $N_G(v) = \{ u\in V(G) \colon vu\in E(G)\}$ is the \emph{neighborhood} of $v$,  and $N_G[v] = N(v) \cup \{v\}$ is the \emph{closed neighborhood} of $v$.
Similarly, given a subset of vertices $S \subseteq V$, $N_G[S]$\ is the set $\cup_{v \in S} N_G[v] $ and $N_G(S)$ is the set $N_G[S] \setminus S$. 
The \emph{degree} of a vertex $v\in V(G)$, denoted by $d_G(v)$, is the number of edges incident with $v$.
A graph $G$ is called \emph{subcubic} if every vertex of $G$ has degree at most $3$.
A \emph{complete bipartite graph} is a graph whose vertices are partitioned into two sets, called \emph{parts}, such that two vertices are adjacent if and only if they belong to different parts.
Given two integers $m,n\ge 0$, a complete bipartite graph whose parts have sizes $m$ and $n$ is denoted by $K_{m,n}$.
In particular, $K_{1,3}$ is the \emph{claw}.
A graph class is \emph{hereditary} if for any graph $G$ in the class, every induced subgraph of $G$ also belongs to the class.

A graph $H$ is a \emph{minor} of a graph $G$ if a graph isomorphic to $H$ can be obtained from $G$ by a sequence of vertex deletions, edge deletions, and edge contractions.
A graph $H$ is an \emph{induced minor} of a graph $G$ if a graph isomorphic to $H$ can be obtained from $G$ by a sequence of vertex deletions and edge contractions.
A \emph{minor model} of a graph $H$ in a graph $G$ is a collection $\{ X_v \}_{v\in V(H)}$ of pairwise disjoint vertex sets $X_v\subseteq V(G)$ called \emph{branch sets} such that each induced subgraph $G[X_v]$ is connected, and if there is an edge $uv\in E(H)$, then there is an edge between $X_u$ and $X_v$ in $G$.
If deleting any vertex from some branch set of $\{ X_v \}_{v\in V(H)}$ results in a collection that is not a minor model of $H$ in $G$, then the minor model $\{ X_v \}_{v\in V(H)}$ is called a \emph{minimal minor model} of a graph $H$ in a graph $G$.
A graph $G$ contains $H$ as a minor if and only if there is a minor model of $H$ in $G$.
An \emph{induced minor model} of $H$ is a minor model with an additional constraint that if $u,v\in V(H)$, $u\neq v$, and $\{ u,v \}\notin E(H)$, then there are no edges between $X_u$ and $X_v$.
Note that a graph $G$ contains $H$ as an induced minor if and only if there is an induced minor model of $H$ in $G$.

A \emph{tree decomposition} of a graph $G$ is a pair $\mathcal{T}=(T, \{ X_t \}_{t\in V(T)})$ where $T$ is a tree and every node $t$ of $T$ is assigned a vertex subset $X_t\subseteq V(G)$ called a \emph{bag} such that the following conditions are satisfied: every vertex $v$ is in at least one bag, for every edge $uv\in E(G)$ there exists a node $t\in V(T)$ such that $X_t$ contains both $u$ and $v$, and for every vertex $u\in V(G)$ the subgraph $T_u$ of $T$ is induced by the set $\{ t\in V(T) \colon u\in X_t \}$ is connected, (that is, a tree).
The \emph{width} of $\mathcal{T}$, denoted by $width(\mathcal{T})$, is the maximum value of $|X_t|-1$ over all $t\in V(T)$.
The \emph{treewidth} of a graph $G$, denoted by $\tw(G)$, is defined as the minimum width of a tree decomposition of $G$.
It is well known that if $H$ is a minor of $G$, then $\tw(H)\le \tw(G)$ (see, e.g.,~\cite{MR1647486}).
Furthermore, subdividing an edge does not change the treewidth of a graph.
A graph class $\mathcal{G}$ is said to be \emph{$(\tw, \omega)$-bounded} if it admits a \emph{$(\tw, \omega)$-binding function}, that is, a function $f:\N \rightarrow \N$ such that for every graph $G\in \mathcal{G}$ and every induced subgraph $G'$ of $G$, we have $\tw(G')\le f(\omega (G'))$.
A graph class that admits a polynomial $(\tw, \omega)$-binding function is said to be \emph{polynomially $(\tw, \omega)$-bounded}.
It is easy to see that every $(\tw, \omega)$-bounded graph class admits a $(\tw, \omega)$-binding function $f$ that is nondecreasing, that is, $k\le \ell$ implies $f(k)\le f(\ell)$ for all $k,\ell\in \N$. 
The same is true for polynomially $(\tw, \omega)$-bounded graph classes.

\begin{observation}\label{poly-nondec}
Every polynomially $(\tw, \omega)$-bounded graph class admits a polynomial $(\tw, \omega)$-binding function that is nondecreasing.
\end{observation}

\begin{proof}
Let $\mathcal{G}$ be a graph class admitting a polynomial $(\tw,\omega)$-binding function $p: \N \rightarrow \N$ given by a polynomial expression $p(x)=\sum_{i=0}^d a_i x^i$ where $a_i\in \mathbb{Z}$ for all $i\in \{0,1,\ldots, d\}$.
Let $r(x)=\sum_{i=0}^d b_i x^i$ be the polynomial that we obtain from $p$ after deleting all terms with negative coefficients.
We claim that $r$ is a nondecreasing $(\tw,\omega)$-binding function.
Observe that $r$ is a $(\tw,\omega)$-binding function for $\mathcal{G}$,  since for every graph $G\in \mathcal{G}$ and every induced subgraph $G'$ of $G$, it holds that $\tw(G')\le p(\omega(G')) \le r(\omega(G'))$.
Furthermore, for every $x\ge 0$ we have $\diff{r}{x} = \sum_{i=1}^d ib_ix^{i-1} \ge 0$, since $b_i\ge 0$ for all~$i$.
Hence, $r$ is nondecreasing on $\N$.
\end{proof}

\section{The main result}
\label{sec:main}

The following result is known as the Grid Minor Theorem, proved by Robertson and Seymour~\cite{zbMATH03963883}.
Since we do not need the precise definition of a $k\times k$ wall, we omit it; we only recall that the $k\times k$ wall is a subcubic graph with treewidth at least $k$.
Hence, any subdivision of a $k \times k$ wall also has treewidth at least $k$.

\begin{theorem}\label{thm:grid-minor thm}
There exists a function $g:\mathbb{N}\to \mathbb{N}$ such that for every $k\in \mathbb{N}$, a graph $G$ has treewidth at least $g(k)$ if and only if $G$ has a subdivision of a $k \times k$ wall as a subgraph.
\end{theorem}

It is shown in~\cite{MR3593966,MR4155282} that $g$ can be chosen to be a polynomial; in particular, $g(k) = \mathcal{O}(k^9\polylog k)$ works.

The following is a basic observation that has been made many times in different contexts with slight variants. 
It is implicit in~\cite{watkinsMesner:cycle}, appears more explicitly in~\cite[Lemma~3.3]{maffray.t:reco} and \cite[Lemma 5.1]{zbMATH07796422}, and with a statement closer to how we formulate it here in~\cite[Lemma~1]{DBLP:conf/iwoca/DallardDHMPT24}.

\begin{lemma}
\label{lemma:classes-STM}
Let $G$ be a graph and $I\subseteq V(G)$ with $|I|\leq 3$.
If there exists a connected component $C$ of $G - I$ such that $I\subseteq N_G(V(C))$, then there exists a connected induced subgraph $H$ of $C$ such that $I\subseteq N_G(V(H))$ and $\omega(H)\leq 3$.
\end{lemma}

\begin{lemma}\label{subcubic minor}
    If $H$ is a subcubic graph and $G$ is a graph that contains $H$ as a minor, then every minimal minor model of $H$ in $G$ is such that each branch set induces a subgraph with clique number at most $3$.
\end{lemma}

\begin{proof}
Let $\{ X_v \}_{v\in V(H)}$ be a minimal minor model of $H$ in $G$.  
Consider a vertex $v\in V(H)$.  Let $J=N_H(v)$. 
For each $u\in J$, pick a vertex $f(u) \in X_u$ such that $f(u)$ has some neighbor in $X_v$.  Let $I = \{f(u)\colon u\in J\}$.  
Since $H$ is subcubic, $|I|\leq 3$.
Let us apply \cref{lemma:classes-STM} to $G[I \cup X_v]$. 
This gives a connected induced subgraph $G'$ of $G[X_v]$ that has neighbors in all $X_u$'s such that $u\in J$.  
By the minimality of the minor model, $X_v=V(G')$, so $\omega(G[X_v]) = \omega(G') \leq 3$.  
\end{proof}

Given a graph class $\mathcal{G}$, we denote by $\mathcal{G}^{im}$ the class of induced minors of graphs in $\mathcal{G}$.
More precisely, \[\mathcal{G}^{im} = \{ H \colon \exists G\in \mathcal{G} \text{ such that } H \text{ is an induced minor of } G\}.\]

\begin{theorem}\label{thm:main}
If $\mathcal{G}$ is a $(\tw,\omega)$-bounded graph class, then $\mathcal{G}^{im}$ is also $(\tw,\omega)$-bounded.
Furthermore, if $\mathcal{G}$ is polynomially $(\tw,\omega)$-bounded, then so is $\mathcal{G}^{im}$.
\end{theorem}

\begin{proof}
Let $f$ be a nondecreasing $(\tw,\omega)$-binding function  for $\mathcal{G}$.
Let $g$ be the function from~\cref{thm:grid-minor thm} and
let $h$ be defined by $h(x) = g(f(3x) + 1)$ for all $x\in \N$. 
We show that $h$ is a $(\tw,\omega)$-binding function for $\mathcal{G}^{im}$.
Since the class $\mathcal{G}^{im}$ is hereditary, it suffices to show that $\tw(H)\le h(\omega(H))$ for all graphs $H\in \mathcal{G}^{im}$.
Suppose for a contradiction that there exists a graph $H\in \mathcal{G}^{im}$ such that $\tw(H)> h(\omega(H))$.
Let $k=f(3\omega(H))+1$.
Then \[\tw(H)> h(\omega(H)) = g (f(3\omega(H))+1)=g(k),\] 
so by~\Cref{thm:grid-minor thm}, $H$ has a subgraph $W$ that is a subdivision of a $k\times k$ wall.

Since $H\in \mathcal{G}^{im}$, there exists a graph $G\in \mathcal{G}$ such that $H$ is an induced minor of $G$.
Let $\{ X_v \}_{v\in V(H)}$ be an induced minor model of $H$ in $G$.
Note that, since $W$ is a subgraph of $H$, the branch sets $X_v$, $v\in V(W)$, form a minor model of $W$ in $G$.
Consequently, there is a minimal minor model $\{ X'_v \}_{v\in V(W)}$ of $W$ in $G$ such that $X'_v\subseteq X_v$ for each vertex $v\in V(W)$.
By \cref{subcubic minor}, $\omega(G[X'_v]) \le 3$ for every $v\in V(W)$.

Let $G'$ be the subgraph of $G$ induced by $\bigcup_{v\in V(W)}X'_v$.
We claim that $\omega (G')\leq 3\omega(H)$.
Let $K$ be a largest clique in $G'$. 
Since for each $v\in V(W)$, every vertex in $K\cap X_v$ belongs to $X_v'$, we have $|K\cap X_v| = |K\cap X_v'|\le \omega(G[X'_v])\le 3$.
Since $K\cap X_v = \emptyset$ for each $v\in V(H)\setminus V(W)$, we conclude that $K$ contains at most $3$ vertices from each branch set of $H$.
Let $C$ be the set of all vertices $v\in V(H)$ such that $X_v$ contains a vertex from $K$.
Then, for every two distinct vertices $u$ and $v$ in $C$, every vertex in $K\cap X_u$ is adjacent to every vertex in $K\cap X_v$, and, hence, $u$ and $v$ are adjacent in $H$ since $\{ X_v \}_{v\in V(H)}$ is an induced minor model of $H$.
It follows that $C$ is a clique in $H$.
This shows that $\omega(G')=|K| = \sum_{v\in C}|K\cap X_v|\leq 3|C|\le 3\omega(H)$, as claimed.

Since $G'$ is an induced subgraph of $G\in \mathcal{G}$ and $f$ is nondecreasing, we have that $\tw (G')\le f (\omega (G'))\le f(3\omega(H))$.
However, we also know that $G'$ contains $W$ as a minor, and thus $\tw(G')\geq \tw(W)\ge k=f (3\omega(H))+1$, a contradiction.
This shows that $h$ is a $(\tw,\omega)$-binding function for $\mathcal{G}^{im}$.

Finally, since the function $g$ from~\cref{thm:grid-minor thm} can be chosen to be a polynomial (see~\cite{MR3593966,MR4155282}) and polynomials compose, $h$ is a polynomial whenever $f$ is (here we also use~\cref{poly-nondec}).
\end{proof}

\section{Concluding remarks}
\label{sec:open}

Following the proof of Theorem~\ref{thm:main}, it is easy to see that if $G$ contains $K_{\ell, \ell}$ as an induced minor with $\ell$ large enough, then it contains a $K_7$-free induced subgraph of high treewidth. 
In fact, Chudnovsky, Hajebi, and Spirkl recently showed in \cite{CHS24} that if a graph $G$ contains $K_{\ell, \ell}$ as an induced minor with $\ell$ large enough, then either $G$ has an induced minor isomorphic to a large wall, or $G$ contains a large complete bipartite induced minor model such that on one side of the bipartition, each branch set is a singleton, and on the other side, each branch set induces a path. 
An immediate consequence of this result is the following.

\begin{theorem}
For every positive integer $k$ there exists a positive integer $\ell$ if a graph $G$ contains $K_{\ell, \ell}$ as an induced minor, then $G$ contains a $K_4$-free induced subgraph with treewidth at least $k$.
\end{theorem}

Note that this is the best possible.

\begin{observation}
For every positive integer $\ell$ there exists a graph $G_\ell$ that contains $K_{\ell, \ell}$ as an induced minor and every $K_3$-free induced subgraph of $G_\ell$ has treewidth at most~$2$.
\end{observation}

\begin{proof}
Let us first note that every graph $G$ is an induced minor of the graph $L(S(G))$, the line graph of the graph obtained from $G$ by subdividing every edge exactly once.

Now, for every positive integer $\ell$, let $G_\ell$ be the graph $L(S(K_{\ell,\ell}))$.
Then, $G_\ell$ contains $K_{\ell, \ell}$ as an induced minor.
Furthermore, since $G_\ell$ is a line graph, it is claw-free. 
Hence, every triangle-free induced subgraph of $G_\ell$ has maximum degree at most~$2$, and is therefore a disjoint union of paths and cycles, so has treewidth at most~$2$.
\end{proof}

\section*{Acknowledgements}

The authors are grateful to \'Edouard Bonnet for helpful discussions.
Claire Hilaire is supported in part by the Slovenian Research and Innovation Agency (research project J1-4008).
Martin Milani\v{c} is supported in part by the Slovenian Research and Innovation Agency (I0-0035, research program P1-0285 and research projects J1-3001, J1-3002, J1-3003, J1-4008, J1-4084, and N1-0370) and by the research program CogniCom (0013103) at the University of Primorska.
Nicolas Trotignon is partially supported by the French National Research Agency under research grant ANR DIGRAPHS ANR-19-CE48-0013-01,  and the LABEX MILYON  (ANR-10-LABX-0070) of Université de Lyon, within the program Investissements d’Avenir (ANR-11-IDEX-0007) operated by the French National Research Agency (ANR).


\begin{thebibliography}{10}

\bibitem{zbMATH07796422}
T.~Abrishami, M.~Chudnovsky, C.~Dibek, S.~Hajebi, P.~Rz{\k{a}}{\.z}ewski,
  S.~Spirkl, and K.~Vu{\v{s}}kovi{\'c}.
\newblock Induced subgraphs and tree decompositions {II}. {Toward} walls and
  their line graphs in graphs of bounded degree.
\newblock {\em J. Comb. Theory, Ser. B}, 164:371--403, 2024.

\bibitem{MR1647486}
H.~L. Bodlaender.
\newblock A partial {$k$}-arboretum of graphs with bounded treewidth.
\newblock {\em Theoret. Comput. Sci.}, 209(1-2):1--45, 1998.

\bibitem{MR3593966}
C.~Chekuri and J.~Chuzhoy.
\newblock Polynomial bounds for the grid-minor theorem.
\newblock {\em J. ACM}, 63(5):Art. 40, 65, 2016.

\bibitem{CHS24}
M.~Chudnovsky, S.~Hajebi, and S.~Spirkl.
\newblock Induced subgraphs and tree decompositions{ XVI. C}omplete bipartite
  induced minors.
\newblock {\em \rm Preprint available at
  \url{https://arxiv.org/abs/2410.16495}}, 2024.

\bibitem{chuTro24}
M.~Chudnovsky and N.~Trotignon.
\newblock On treewidth and maximum cliques.
\newblock {\em \rm Preprint available at
  \url{https://arxiv.org/abs/2405.07471}}, 2024.

\bibitem{MR4155282}
J.~Chuzhoy and Z.~Tan.
\newblock Towards tight(er) bounds for the excluded grid theorem.
\newblock {\em J. Combin. Theory Ser. B}, 146:219--265, 2021.

\bibitem{DBLP:conf/iwoca/DallardDHMPT24}
C.~Dallard, M.~Dumas, C.~Hilaire, M.~Milani{\v c}, A.~Perez, and N.~Trotignon.
\newblock Detecting {$K_{2,3}$} as an induced minor.
\newblock In A.~A. Rescigno and U.~Vaccaro, editors, {\em Combinatorial
  Algorithms - 35th International Workshop, {IWOCA} 2024, Ischia, Italy, July
  1-3, 2024, Proceedings}, volume 14764 of {\em Lecture Notes in Computer
  Science}, pages 151--164. Springer, 2024.

\bibitem{DBLP:journals/jctb/DallardMS24}
C.~Dallard, M.~Milani{\v c}, and K.~{\v S}torgel.
\newblock Treewidth versus clique number. {II.} {T}ree-independence number.
\newblock {\em Journal of Combinatorial Theory, Series {B}}, 164:404--442,
  2024.

\bibitem{maffray.t:reco}
F.~Maffray and N.~Trotignon.
\newblock Algorithms for perfectly contractile graphs.
\newblock {\em SIAM Journal on Discrete Mathematics}, 19(3):553--574, 2005.

\bibitem{zbMATH03963883}
N.~Robertson and P.~D. Seymour.
\newblock Graph minors. {V}. {Excluding} a planar graph.
\newblock {\em J. Comb. Theory, Ser. B}, 41:92--114, 1986.

\bibitem{watkinsMesner:cycle}
M.~E. Watkins and D.~M. Mesner.
\newblock Cycles and connectivity in graphs.
\newblock {\em Canadian Journal of Mathematics}, 19:1319--1328, 1967.

\end{thebibliography}
\end{document}